\documentclass[12pt,english]{article}
\usepackage{mathptmx}

\usepackage[T1]{fontenc}
\usepackage[latin9]{inputenc}
\usepackage[a4paper]{geometry}
\usepackage{babel}
\usepackage{amsmath}
\usepackage{amsthm}
\usepackage{amssymb}
\usepackage{setspace}
\usepackage[numbers,square,sort,comma,numbers]{natbib}
\usepackage[all]{xy}
\usepackage[unicode=true,
 bookmarks=true,bookmarksnumbered=false,bookmarksopen=false,
 breaklinks=false,pdfborder={0 0 1},backref=false,colorlinks=false]
 {hyperref}
\hypersetup{pdftitle={Use Lyx for typesetting: A template for empirical research projects},
 pdfauthor={First Author and Second Author}}

\makeatletter
\newenvironment{lyxlist}[1]
	{\begin{list}{}
		{\settowidth{\labelwidth}{#1}
		 \setlength{\leftmargin}{\labelwidth}
		 \addtolength{\leftmargin}{\labelsep}
		 }}
	{\end{list}}

\hypersetup{
    colorlinks=true,       
    linkcolor=blue,          
    citecolor=blue,        
    filecolor=blue,      
    urlcolor=blue,           
    breaklinks=true
}

\usepackage{booktabs}
\usepackage{longtable}
\usepackage{pdflscape}
\usepackage[square,sort,comma,numbers]{natbib}
\usepackage{amsthm}
\usepackage[title]{appendix}

\makeatother

\begin{document}
\title{\onehalfspacing{}Characterizing the kernel of the Burau representation modulo 2 for
the 4-strand braid group}
\author{Donsung Lee}
\date{September 11, 2023}

\maketitle
\medskip{}

\begin{abstract}
\begin{spacing}{0.9}
\noindent In 1997, Cooper\textendash Long established the nontriviality
of the kernel of the modulo 2 Burau representation for the 4-strand
braid group, $B_{4}$. This paper extends their work by embedding
the image group into a finitely presented group. As an application,
we characterize the kernel as the intersection of $B_{4}$ with a
normal closure of a finite number of elements in an HNN extension
of $B_{4}$.\vspace{1cm}

\noindent \textbf{Keywords:} braid group, Burau representation, HNN
extension.\medskip{}

\noindent \textbf{Mathematics Subject Classification 2020:} 20F05,
20F29, 20F36.
\end{spacing}
\end{abstract}
\noindent $ $\theoremstyle{definition}
\newtheorem{definition}{Definition}[section]
\theoremstyle{remark}
\newtheorem{theorem}{Theorem}[section]
\newtheorem{lemma}[theorem]{Lemma}
\newtheorem{corollary}[theorem]{Corollary}
\newtheorem{remark}[theorem]{Remark}

\newtheorem*{ack}{Acknowledgements}
\newtheorem*{theorem11}{Theorem \textnormal{1.1}}
\newtheorem*{theorem13}{Theorem \textnormal{1.3}}
\newtheorem*{lemma22}{Lemma \textnormal{2.2}}
\newtheorem*{prooflemma21}{Proof of Lemma \textnormal{2.1}}
\newtheorem*{prooflemma22}{Proof of Lemma \textnormal{2.2}}
\newtheorem*{prooftheorem13}{Proof of Theorem \textnormal{1.3}}
\newtheorem*{claim1}{Claim \textnormal{1}}
\newtheorem*{claim2}{Claim \textnormal{2}}
\newtheorem*{claim3}{Claim \textnormal{3}}
\newtheorem*{proofclaim1}{Proof of Claim \textnormal{1}}
\newtheorem*{proofclaim2}{Proof of Claim \textnormal{2}}
\newtheorem*{proofclaim3}{Proof of Claim \textnormal{3}}

\section{Introduction}

The reduced Burau representation $\beta_{n}:B_{n}\to\mathrm{GL}\left(n-1,\,\mathbb{Z}\left[t,t^{-1}\right]\right)$
of the $n$-strand braid group $B_{n}$ is a well-known homological
representation \citep{MR2435235} over the Laurent polynomial ring
of integral coefficients. Its kernel is known to be trivial for $n=2,3$
\citep{MR2435235}, and nontrivial for $n\ge5$ \citep{MR1725480}.
Church and Farb showed in \citep{MR2629766} that $\ker\beta_{n}$
is not finitely generated for $n\ge6$. However, for the 4-strand
braid group $B_{4}$, whether $\ker\beta_{4}$ is trivial or not remains
an open problem.

Until now, much effort has been made to characterize the kernel for
$n\ge4$, or find a subgroup of $B_{n}$ that contains it. We denote
the standard generating set for $B_{n}$ by $\left\{ \sigma_{i}\right\} _{1\le i\le n-1}$.
A classical observation of Gorin\textendash Lin \citep{MR0251712}
is that $\ker\beta_{4}$ is contained in the group normally generated
by the single braid $\sigma_{3}\sigma_{1}^{-1}$. In 2017, Calvez
and Ito deduced in \citep{MR3660095} a Garside-theoretic criterion
for a braid to be in $\ker\beta_{4}$. Recently, Dlugie showed in
\citep{dlugie2022burau} that $\ker\beta_{4}$ is also contained in
the normal closure of $\sigma_{1}^{d}$, for any $d=5,6,$ or $8$.

On the other hand, define $\beta_{4}\otimes\mathbb{F}_{p}:B_{4}\to\mathrm{GL}\left(3,\,\mathbb{F}_{p}\left[t,t^{-1}\right]\right)$
to be the reduced Burau representation modulo $p$ entrywise. Cooper\textendash Long
proved that $\ker\beta_{4}\otimes\mathbb{F}_{p}$ is nontrivial when
$p=2$ \citep{MR1431138} and $p=3$ \citep{MR1668343}. After Cooper\textendash Long,
Lee and Song showed in \citep{MR2175118} that $\ker\beta_{4}\otimes\mathbb{F}_{p}$
consists only of pseudo-Anosov braids.

In this paper, we characterize $\ker\beta_{4}\otimes\mathbb{F}_{2}$
as the normal closure of three elements in an HNN extension of $B_{4}$
(Theorem 1.3). Moreover, we define $\gamma:B_{4}/Z\left(B_{4}\right)\to\mathrm{PGL}\left(3,\,\mathbb{Z}\left[t,t^{-1}\right]\right)$
to be the corresponding homomorphism for $B_{4}/Z\left(B_{4}\right)$,
where $Z\left(B_{4}\right)$ is the center, and then we obtain a parallel
but simpler result; we characterize $\ker\gamma$, as the normal closure
of one element in an HNN extension of $B_{4}/Z\left(B_{4}\right)$
(Theorem 1.1). These theorems are in line with previous results that
have identified subgroups containing $\ker\beta_{4}$ defined as normal
closures. Cooper\textendash Long's presentations \citep{MR1431138}
of the image of $\beta_{4}\otimes\mathbb{F}_{2}$ and $\gamma$ motivated
our results. We combine the results in \citep{MR1431138} with the
elementary group-theoretic methods (the Tietze transformations and
the construction of the HNN extension). A key step is Lemma 2.1 (resp.
Lemma 3.1) in Section 2 (resp. Section 3).

Additionally, one obtains the \emph{integral Burau representation}
by evaluating $t$ in $\beta_{n}$ at $-1$. It is worth mentioning
that Smythe showed in \citep{MR545151} that the kernel of the integral
Burau representation is isomorphic to $\mathbb{Z}\times F_{\infty}$,
where $F_{\infty}$ is the free group of countably infinite rank.
In general, Brendle\textendash Margalit\textendash Putman \citep{MR3323579}
found an explicit generating set for its kernel for each $n$.

To state our results in precise terms, we recall several group-theoretic
facts on the 4-strand braid group $B_{4}$. From its center $Z\left(B_{4}\right)$,
define the quotient group
\begin{align*}
Q_{4} & :=B_{4}/Z\left(B_{4}\right).
\end{align*}

Then, we have the following presentations \citep{MR1431138}:
\begin{align*}
 & B_{4}\cong\left\langle x,y\;|\;1=x^{4}y^{-3}=\left[x^{2},yxy\right]\right\rangle ,\\
 & Q_{4}\cong\left\langle x,y\;|\;1=x^{4}=y^{3}=\left[x^{2},yxy\right]\right\rangle ,
\end{align*}

\noindent where we use the commutator convention $\left[a,b\right]=a^{-1}b^{-1}ab$
and $x$ (resp. $y$) corresponds to an element $\sigma_{1}\sigma_{2}\sigma_{3}$
(resp. $\sigma_{1}\sigma_{2}\sigma_{3}\sigma_{1}$) in $B_{4}$ or
the image in $Q_{4}$. Let $\pi:B_{4}\to Q_{4}$ be the quotient map.

Denote by $B_{4,2}$ the image of the representation $\beta_{4}\otimes\mathbb{F}_{2}:B_{4}\to\mathrm{GL}\left(3,\,\mathbb{F}_{2}\left[t,t^{-1}\right]\right)$.
For the natural surjection $p:\mathrm{GL}\left(3,\,\mathbb{F}_{2}\left[t,t^{-1}\right]\right)\to\mathrm{PGL}\left(3,\,\mathbb{F}_{2}\left[t,t^{-1}\right]\right)$,
we also define $Q_{4,2}$ to be $p\left(B_{4,2}\right)$ in $\mathrm{PGL}\left(3,\,\mathbb{F}_{2}\left[t,t^{-1}\right]\right)$.

The center $Z\left(B_{4}\right)$ of $B_{4}$ is generated by the
element $\Delta^{2}$, where
\begin{align*}
\Delta & :=\sigma_{1}\sigma_{2}\sigma_{3}\sigma_{1}\sigma_{2}\sigma_{1}.
\end{align*}

For the identity matrix $I_{3}$, a direct computation yields
\begin{align}
\beta_{4}\left(\Delta^{2}\right) & =t^{4}I_{3}\in\ker p.
\end{align}

From (1), we define a group homomorphism $\gamma:Q_{4}\to\mathrm{PGL}\left(3,\,\mathbb{F}_{2}\left[t,t^{-1}\right]\right)$
by 
\begin{align*}
\gamma\left(bZ\left(B_{4}\right)\right) & :=p\left(\beta_{4}\otimes\mathbb{F}_{2}\left(b\right)\right),\;b\in B_{4},
\end{align*}

\noindent where we denote by $bZ\left(B_{4}\right)$ the coset of
$b$ in $B_{4}$.

Then, we directly see the following diagram
\begin{align*}
\xymatrix{B_{4}\ar[d]^{\pi}\ar[r]^{\beta_{4}\otimes\mathbb{F}_{2}} & B_{4,2}\ar[d]^{p|_{B_{4,2}}}\\
Q_{4}\ar[r]^{\gamma} & Q_{4,2}
}
\end{align*}

\noindent commutes.

Abusing notation, we denote by $x$ (resp. $y$) the image of $\sigma_{1}\sigma_{2}\sigma_{3}\in B_{4}$
(resp. $\sigma_{1}\sigma_{2}\sigma_{3}\sigma_{1}\in B_{4}$) in $B_{4}$,
$B_{4,2}$, $Q_{4}$, or $Q_{4,2}$ as before. For a list $W$ of
elements in a group $A$, we denote by $\left\langle W\right\rangle _{A}$
(resp. $\left\langle \!\!\left\langle W\right\rangle \!\!\right\rangle _{A}$)
the subgroup generated (resp. normally generated) by the elements
in $W$. We omit the subscript $A$ when there is no room for confusion.
Let us define $G$ to be the HNN extension of $Q_{4}$ constructed
from the identity map from $\left\langle yxy\right\rangle \subset Q_{4}$
to itself, and denote by $\epsilon:Q_{4}\hookrightarrow G$ the associated
injection. In other words, by abuse of notation, $G$ is presented
by
\begin{align*}
\left\langle x,y,t\;\left|\;\begin{aligned}1=x^{4}=y^{3}=\left[x^{2},yxy\right],\\
1=\left[t,yxy\right].
\end{aligned}
\right.\right\rangle ,
\end{align*}

\noindent where $\epsilon\left(x\right)=x$ and $\epsilon\left(y\right)=y$.
For brevity, we adopt the convention identifying $Q_{4}$ with its
image $\epsilon\left(Q_{4}\right)\subset G$ in the statement of our
main theorem:

\noindent \begin{theorem}

\noindent The kernel of the homomorphism $\gamma:Q_{4}\to Q_{4,2}$
satisfies the identity
\begin{align}
\ker\gamma & =Q_{4}\cap\left\langle \!\!\left\langle \left[yxy,x\right]tx^{-1}t^{-1}\right\rangle \!\!\right\rangle _{G}.
\end{align}

\noindent \end{theorem}

The left hand side of (2) is defined in an arithmetic way and the
other side is defined in group-theoretic terms. It is well-known that
$\ker\gamma$ is nontrivial, which Cooper and Long showed in \citep{MR1431138}.
We recover this result from Theorem 1.1, by illustrating the right
hand side of (2) is nontrivial.

\noindent \begin{corollary}

\noindent The group $\ker\gamma$ is nontrivial.

\noindent \end{corollary}
\begin{proof}

\noindent Temporarily define $\delta:=txt^{-1},\;\alpha:=\left[yxy,x\right]tx^{-1}t^{-1}.$
The right hand side of (2) includes an element $\left[yxy,x\right]^{4}$
in $Q_{4}$. Indeed, a direct computation yields:
\begin{align*}
 & \,\left[yxy,x\right]^{4}\\
= & \,\left(\left[yxy,x\right]tx^{-1}t^{-1}\right)\left(txt^{-1}\left[yxy,x\right]tx^{-2}t^{-1}\right)\left(tx^{2}t^{-1}\left[yxy,x\right]tx^{-3}t^{-1}\right)\left(tx^{3}t^{-1}\left[yxy,x\right]\right)\\
= & \,\alpha\left(\delta\alpha\delta^{-1}\right)\left(\delta^{2}\alpha\delta^{-2}\right)\left(tx^{3}t^{-1}\left[yxy,x\right]tx^{-4}t^{-1}\right)\\
= & \,\alpha\left(\delta\alpha\delta^{-1}\right)\left(\delta^{2}\alpha\delta^{-2}\right)\left(\delta^{3}\alpha\delta^{-3}\right),
\end{align*}

\noindent where we used the relation $x^{4}=1$ in $Q_{4}$.

We only need to check that $\left[yxy,x\right]^{4}$ is nontrivial
in $Q_{4}$. This fact directly follows by computing the images of
determinant 1 of the representatives $\left[yxy,x\right]^{4}$ in
$B_{4}$ under $\beta_{4}$.$\qedhere$

\noindent \end{proof}

We state an analogue of Theorem 1.1 for the representation $\beta_{4}\otimes\mathbb{F}_{2}:B_{4}\to B_{4,2}$
mentioned in the title. Let us define $\widetilde{G}$ to be the HNN
extension of $B_{4}$ constructed from the identity map from the subgroup
$\left\langle yxy,\Delta^{2}\right\rangle \subset B_{4}$ to itself,
and denote by $\widetilde{\epsilon}:B_{4}\hookrightarrow\widetilde{G}$
the associated injection. In other words, $\widetilde{G}$ is presented
as follows:
\begin{align*}
\left\langle x,y,t\;\left|\;\begin{aligned}1=x^{4}y^{-3}=\left[x^{2},yxy\right],\\
1=\left[t,yxy\right]=\left[t,x^{4}\right].
\end{aligned}
\right.\right\rangle ,
\end{align*}

\noindent where $\widetilde{\epsilon}\left(x\right)=x$ and $\widetilde{\epsilon}\left(y\right)=y$.
For brevity, we define $x_{1}$ to be $\left[yxy,x\right]$ as a word
in $x$ and $y$, and adopt the convention identifying $B_{4}$ with
its image $\widetilde{\epsilon}\left(B_{4}\right)\subset\widetilde{G}$.

\noindent \begin{theorem}

\noindent The kernel of the representation $\beta_{4}\otimes\mathbb{F}_{2}:B_{4}\to B_{4,2}$
satisfies the identity
\begin{align*}
\ker\left(\beta_{4}\otimes\mathbb{F}_{2}\right) & =B_{4}\cap\left\langle \!\!\left\langle x_{1}^{4},\:\left[x_{1}^{2},yxy\right],\:\left[yxy,x_{1}\right]tx_{1}^{-1}t^{-1}\right\rangle \!\!\right\rangle _{\widetilde{G}}.
\end{align*}

\noindent \end{theorem}

The rest of this paper is organized as follows. As we mentioned earlier,
Lemma 2.1 (resp. Lemma 3.1) plays a key role in the proof of the main
result, Theorem 1.1 (resp. Theorem 1.3). In Section 2, we prove Lemma
2.1 and Theorem 1.1. In Section 3, we prove Lemma 3.1 and Theorem
1.3.

\noindent \begin{ack}

\noindent This paper is supported by National Research Foundation
of Korea (grant number 2020R1C1C1A01006819) and Samsung Science and
Technology Foundation (project number SSTF-BA2001-01). The author
is grateful to the advisor Dohyeong Kim for various valuable comments
on the organization and revision of this paper.

\noindent \end{ack}

\section{Proof of Theorem 1.1\label{sec:Proof-of-Theorem-1.1}}

For each nonnegative integer $i$, let us define a family $\left\{ G_{i}\right\} _{i\ge0}$
of finitely presented groups by the presentation:
\begin{align*}
G_{i} & :=\left\langle x,y,t\;\left|\;\begin{aligned}1=x_{j}^{4}=y^{3}=\left[x_{j}^{2},yxy\right],\:j\in\left\{ 0,1,\cdots,i\right\} ,\\
1=\left[t,yxy\right],\:tx_{i}t^{-1}=\left[yxy,x_{i}\right].
\end{aligned}
\right.\right\rangle ,
\end{align*}

\noindent where $x_{i}$ is a word in $x$ and $y$, recursively defined
as
\begin{align}
x_{0} & :=x,\;x_{i+1}:=\left[yxy,x_{i}\right],\;0\le i.
\end{align}

We state the key lemma. Recall that $Q_{4,2}$ is generated by $x$
and $y$.

\noindent \begin{lemma}

\noindent For each nonnegative integer $i$, there is an injective
homomorphism $\epsilon_{i}:Q_{4,2}\to G_{i}$ satisfying
\begin{align*}
\epsilon_{i}\left(x\right) & =x,\;\epsilon_{i}\left(y\right)=y.
\end{align*}

\noindent \end{lemma}

Our proof of Lemma 2.1 relies on Lemma 2.2 and Lemma 2.3.

\noindent \begin{lemma}

\noindent The group $Q_{4,2}$ is presented by the generators $x,y$
and the following relations:
\begin{lyxlist}{00.00.0000}
\item [{$\left.1_{o}\right)$}] \noindent $1=y^{3},$
\item [{$\left.1_{x}\right)$}] \noindent $1=x_{i}^{4},$ for each integer
$i\ge0,$
\item [{$\left.2_{x}\right)$}] \noindent $1=\left[x_{i}^{2},yxy\right],$
for each integer $i\ge0,$
\end{lyxlist}
\noindent where $x_{i}$ is a word in $x$ and $y$ defined by (3).

\noindent \end{lemma}
\begin{proof}

See Appendix A, where we utilize the presentation for $Q_{4,2}$ of
Cooper and Long in \citep{MR1431138} to prove Lemma 2.2.$\qedhere$

\noindent \end{proof}
\begin{lemma}

\noindent For each nonnegative integer $i$, let us define a subgroup
$S_{i}$ to be generated by two elements $x_{i}$ and $yxy$ in $Q_{4,2}$.
Then, for every pair $\left(i,j\right)$, a map $\eta_{ij}$ between
the generating sets
\begin{align*}
\eta_{ij}\left(x_{i}\right) & =x_{j},\;\eta_{ij}\left(yxy\right)=yxy,
\end{align*}

\noindent extends to the group isomorphism $\eta_{ij}:S_{i}\to S_{j}$.

\noindent \end{lemma}
\begin{proof}

\noindent The proof depends on the fact that $Q_{4,2}$ is a subgroup
of $\mathrm{PGL}\left(3,\,\mathbb{F}_{2}\left(t\right)\right)$, as
well as $\mathrm{PGL}\left(3,\,\mathbb{F}_{2}\left[t,t^{-1}\right]\right)$.
Suppose there exists a matrix $M\in\mathrm{GL}\left(3,\,\mathbb{F}_{2}\left(t\right)\right)$
such that
\begin{align}
\begin{aligned} & p\left(M\right)xp\left(M\right)^{-1}=\left[yxy,x\right],\\
 & p\left(M\right)\left(yxy\right)p\left(M\right)^{-1}=yxy,
\end{aligned}
\end{align}

\noindent where $p:\mathrm{GL}\left(3,\,\mathbb{F}_{2}\left(t\right)\right)\to\mathrm{PGL}\left(3,\,\mathbb{F}_{2}\left(t\right)\right)$
is the natural surjection, mentioned in Section 1. Then, for every
pair $\left(i,j\right)$, we define $\eta_{ij}:S_{i}\to S_{j}$ by
conjugation, for each $A\in S_{i}$:
\begin{align*}
\eta_{ij}\left(A\right) & =p\left(M\right)^{j-i}Ap\left(M\right)^{i-j},
\end{align*}

\noindent which is proved inductively as follows. Suppose
\begin{align*}
p\left(M\right)x_{i}p\left(M\right)^{-1} & =x_{i+1}
\end{align*}

\noindent holds for every $i$ such that $0\le i\le N$. Then, by
the recursive definition of $x_{N+1}$,
\begin{align*}
 & \,p\left(M\right)x_{N+1}p\left(M\right)^{-1}\\
= & \,p\left(M\right)\left(yxy\right)^{-1}x_{N}^{-1}\left(yxy\right)x_{N}p\left(M\right)^{-1}\\
= & \,\left(yxy\right)^{-1}x_{N+1}^{-1}\left(yxy\right)x_{N+1}=x_{N+2}.
\end{align*}

Now, we need only to show there exists such a matrix $M$ satisfying
(4). Choose
\begin{align*}
M & =\left(\begin{array}{ccc}
1 & t & t\\
\frac{t^{3}}{\left(1+t\right)^{2}} & \frac{t}{\left(1+t\right)^{2}} & \frac{t^{3}}{\left(1+t\right)^{2}}\\
\frac{1+t+t^{3}}{\left(1+t\right)^{2}} & \frac{t\left(1+t+t^{2}\right)}{\left(1+t\right)^{2}} & \frac{t^{2}}{\left(1+t\right)^{2}}
\end{array}\right).
\end{align*}

Then, for $x,y\in B_{4,2}\subset\mathrm{GL}\left(3,\,\mathbb{F}_{2}\left(t\right)\right)$
a direct computation yields (see the matrices at the beginning in
Appendix A)
\begin{align*}
 & MxM^{-1}=t\left[yxy,x\right],\\
 & M\left(yxy\right)M^{-1}=yxy.\qedhere
\end{align*}

\noindent \end{proof}
\begin{prooflemma21}

\noindent From the isomorphism in Lemma 2.3
\begin{align*}
 & \eta_{i\left(i+1\right)}:S_{i}\to S_{i+1},\\
 & \eta_{i\left(i+1\right)}\left(x_{i}\right)=x_{i+1},\;\eta_{i\left(i+1\right)}\left(yxy\right)=yxy,
\end{align*}

\noindent we construct the HNN extension $H_{i}$ of $Q_{4,2}$. Using
the presentation for $Q_{4,2}$ in Lemma 2.2, $H_{i}$ is presented
by the generators $x,y,t$ and the following relations:
\begin{lyxlist}{00.00.0000}
\item [{$\left.1_{o}\right)$}] \noindent $1=y^{3},$
\item [{$\left.1_{x}\right)$}] \noindent $1=x_{j}^{4},$ for each integer
$j\ge0,$
\item [{$\left.2_{x}\right)$}] \noindent $1=\left[x_{j}^{2},yxy\right],$
for each integer $j\ge0,$
\item [{$\left.t_{i}\right)$}] \noindent $tx_{i}t^{-1}=\left[yxy,x_{i}\right],\;1=\left[t,yxy\right].$
\end{lyxlist}
$\quad\;\:$Then, from $\left.t_{i}\right)$ we obtain
\begin{align}
t^{j-i}x_{i}t^{i-j} & =x_{j},
\end{align}

\noindent for each integer $j$ larger than $i$, in the same inductive
manner as in the proof of Lemma 2.3. Thus, for every integer $j$
such that $j>i$, we have
\begin{align*}
x_{j}^{4} & =t^{j-i}x_{i}^{4}t^{i-j}=1,
\end{align*}

\noindent where we have taken both sides of (5) to the fourth power
and applied $\left.1_{x}\right)$ for $i$, and 
\begin{align*}
 & \,\left[x_{j}^{2},yxy\right]\\
= & \,t^{j-i}x_{i}^{-2}t^{i-j}\left(yxy\right)^{-1}t^{j-i}x_{i}^{2}t^{i-j}\left(yxy\right)\\
= & \,t^{j-i}\left[x_{i}^{2},yxy\right]t^{i-j}=1,
\end{align*}

\noindent where the first equality follows from (5), the second from
$\left.t_{i}\right)$, and the third from $\left.2_{x}\right)$ for
$i$.

Thus, the relations in $\left.1_{x}\right)$ and $\left.2_{x}\right)$
for $j$ such that $j>i$ are redundant. After deleting them, we have
the following presentation for $H_{i}$, which is the same as that
for $G_{i}$,
\begin{align*}
\left\langle x,y,t\;\left|\;\begin{aligned}1=x_{j}^{4}=y^{3}=\left[x_{j}^{2},yxy\right],\:j\in\left\{ 0,1,\cdots,i\right\} ,\\
1=\left[t,yxy\right],\:tx_{i}t^{-1}=\left[yxy,x_{i}\right].
\end{aligned}
\right.\right\rangle ,
\end{align*}

\noindent as defined in the first section.$\qed$

\noindent \end{prooflemma21}
\begin{remark}

Before proving Theorem 1.1, we comment briefly on Lemma 2.1. Generally,
Higman's embedding theorem \citep{MR130286} guarantees that a recursively
presented group, such as $Q_{4,2}$, is embedded into a finitely presented
group in principle. However, finding a computationally useful embedding
is generally a complex task, as exemplified in \citep{MR4478033}.
Concerning the embedding problem on $Q_{4,2}$ we are dealing with,
a helpful way to understand the group $G_{i}$ is to regard it as
a supergroup of the standard lamplighter group $L$. Indeed, the subgroup
$\left\langle x_{i}^{2},t\right\rangle $ of $G_{i}$ is isomorphic
to $L$ (cf. Claim 3 in $\mathrm{Appendix}\;\mathrm{A}$).

\noindent \end{remark}

We return to the proof of Theorem 1.1. In Section 1, we defined the
HNN extension $G$ of $Q_{4}$ by the identity map from $\left\langle yxy\right\rangle \subset Q_{4}$
to itself with the associated injection $\epsilon:Q_{4}\hookrightarrow G.$

\noindent \begin{theorem11}

\noindent The kernel of the homomorphism $\gamma:Q_{4}\to Q_{4,2}$
satisfies the identity
\begin{align*}
\ker\gamma & =Q_{4}\cap\left\langle \!\!\left\langle \left[yxy,x\right]tx^{-1}t^{-1}\right\rangle \!\!\right\rangle _{G}.
\end{align*}
\end{theorem11}
\begin{proof} 

\noindent By construction, the extended group $G$ has the following
presentation:
\begin{align*}
\left\langle x,y,t\;\left|\;\begin{aligned}1=x^{4}=y^{3}=\left[x^{2},yxy\right],\\
1=\left[t,yxy\right].
\end{aligned}
\right.\right\rangle ,
\end{align*}

\noindent and recall that $G_{0}$ is presented by
\begin{align*}
\left\langle x,y,t\;\left|\;\begin{aligned}1=x^{4}=y^{3}=\left[x^{2},yxy\right],\\
1=\left[t,yxy\right],\:txt^{-1}=\left[yxy,x\right].
\end{aligned}
\right.\right\rangle .
\end{align*}

Comparing the two presentations, we have
\begin{align*}
G_{0} & \cong G/\left\langle \!\!\left\langle \left[yxy,x\right]tx^{-1}t^{-1}\right\rangle \!\!\right\rangle _{G},
\end{align*}

\noindent and define $\pi_{0}:G\to G_{0}$ to be the quotient map.

By Lemma 2.1, we take the injection $\epsilon_{0}:Q_{4,2}\to G_{0}$
carrying $x$ (resp. $y$) to $x$ (resp. $y$). By the HNN construction,
$\epsilon$ also carries $x$ (resp. $y$) to $x$ (resp. $y$). Therefore,
the following diagram
\begin{align*}
\xymatrix{Q_{4}\ar@{^{(}->}[d]^{\epsilon}\ar[r]^{\gamma} & Q_{4,2}\ar@{^{(}->}[d]^{\epsilon_{0}}\\
G\ar[r]^{\pi_{0}} & G_{0}
}
\end{align*}

\noindent commutes, and we have
\begin{align*}
\epsilon\left(\ker\gamma\right) & =\epsilon\left(Q_{4}\right)\cap\ker\pi_{0}.
\end{align*}

Now the result desired follows, since by the definition of $\pi_{0}$,
\begin{align*}
\ker\pi_{0} & =\left\langle \!\!\left\langle \left[yxy,x\right]tx^{-1}t^{-1}\right\rangle \!\!\right\rangle _{G}.\qedhere
\end{align*}

\noindent \end{proof}
\begin{remark}

\noindent There exists a natural relation among the groups $G_{i}$.
For each pair $\left(i,j\right)$ such that $0\le i\le j$, define
a map $\Sigma_{ij}:G_{j}\to G_{i}$ by
\begin{align*}
\Sigma_{ij}\left(x\right) & :=x,\:\Sigma_{ij}\left(y\right):=y,\:\Sigma_{ij}\left(t\right):=t,
\end{align*}

\noindent which is well-defined and surjective by definition. Then,
the groups $\left\{ G_{i}\right\} _{i\ge0}$ and the maps $\left\{ \Sigma_{ij}\right\} _{i\le j}$
forms an inverse system. Combining this with Lemma 2.1, it directly
follows that the following diagram
\begin{align*}
\xymatrix{ & G_{j}\ar[d]^{\Sigma_{ij}}\\
Q_{4,2}\ar@{^{(}->}[ur]^{\epsilon_{j}}\ar@{^{(}->}[r]^{\epsilon_{i}} & G_{i}
}
\end{align*}

\noindent always commutes for each pair $\left(i,j\right)$ such that
$0\le i\le j$.

\noindent \end{remark}
\begin{remark}

\noindent One may wonder where the matrix $M$ is from in the proof
of Lemma 2.3. For any rational function $f\in\mathbb{F}_{2}\left(t\right)$,
the matrix $M_{f}\in\mathrm{GL}\left(3,\,\mathbb{F}_{2}\left(t\right)\right)$
defined by
\begin{align*}
M_{f} & :=\left(\begin{array}{ccc}
f & f+1+t & f+1+t\\
\frac{t^{3}}{\left(1+t\right)^{2}} & \frac{t}{\left(1+t\right)^{2}} & \frac{t^{3}}{\left(1+t\right)^{2}}\\
\frac{f\left(1+t^{2}\right)+\left(t+t^{2}+t^{3}\right)}{\left(1+t\right)^{2}} & \frac{f\left(1+t^{2}\right)+\left(1+t+t^{3}\right)}{\left(1+t\right)^{2}} & \frac{f\left(1+t^{2}\right)+1}{\left(1+t\right)^{2}}
\end{array}\right)
\end{align*}

\noindent satisfies the relation
\begin{align*}
 & M_{f}xM_{f}^{-1}=t\left[yxy,x\right],\\
 & M_{f}\left(yxy\right)M_{f}^{-1}=yxy,
\end{align*}

\noindent where $M_{f}$ is always invertible, since $\det M_{f}=t^{2}$.
We chose $M$ as $M_{1}$ above. Conversely, such matrices should
be of the form $M_{f}$ for a rational function $f$ up to scalar
multiplication by solving simultaneous equations on the matrix entries.
In particular, we cannot realize $\eta_{ij}$ as the restriction of
an inner automorphism of $\mathrm{PGL}\left(3,\,\mathbb{F}_{2}\left[t,t^{-1}\right]\right).$
In addition, we should mention that $f$ can be chosen such that $M_{f}$
satisfies even the unitarity, a well-known property satisfied by the
image of the Burau representation.

\noindent \end{remark}

\section{Proof of Theorem 1.3\label{sec:Proof-of-Theorem-1.3}}

We give a proof of Theorem 1.3 which is similar to that of Theorem
1.1. In order to avoid repetition in arguments, we concentrate on
the differences. We defined the HNN extension $\widetilde{\epsilon}:B_{4}\hookrightarrow\widetilde{G}$
by the identity map from $\left\langle yxy,\Delta^{2}\right\rangle \subset B_{4}$
to itself in Section 1. By construction, the extended group $\widetilde{G}$
has the following presentation
\begin{align*}
\left\langle x,y,t\;\left|\;\begin{aligned}1=x^{4}y^{-3}=\left[x^{2},yxy\right],\\
1=\left[t,yxy\right]=\left[t,x^{4}\right].
\end{aligned}
\right.\right\rangle .
\end{align*}

Let us define a family $\left\{ \widetilde{G_{i}}\right\} $ of finitely
presented groups by the presentation for each positive integer $i$
as
\begin{align*}
\widetilde{G_{i}} & :=\left\langle x,y,t\;\left|\;\begin{aligned}1=x^{4}y^{-3}=\left[x^{2},yxy\right]=x_{j}^{4}=\left[x_{j}^{2},yxy\right],\:j\in\left\{ 1,\cdots,i\right\} ,\\
1=\left[t,yxy\right]=\left[t,x^{4}\right],\:tx_{i}t^{-1}=\left[yxy,x_{i}\right].
\end{aligned}
\right.\right\rangle ,
\end{align*}

\noindent where $x_{i}$ is a word in $x$ and $y$ defined by (3).
As in the proof of Theorem 1.1 in Section 2, we need the following
lemma to obtain Theorem 1.3.

\noindent \begin{lemma}

\noindent For each positive integer $i$, there is an injective homomorphism
$\widetilde{\epsilon_{i}}:B_{4,2}\to\widetilde{G_{i}}$ satisfying
\begin{align*}
\widetilde{\epsilon_{i}}\left(x\right) & :=x,\;\widetilde{\epsilon_{i}}\left(y\right):=y.
\end{align*}

\noindent \end{lemma}
\begin{prooftheorem13}

\noindent The first member of $\left\{ \widetilde{G_{i}}\right\} $,
$\widetilde{G_{1}}$ is presented by
\begin{align*}
\widetilde{G_{1}} & :=\left\langle x,y,t\;\left|\;\begin{aligned}1=x^{4}y^{-3}=\left[x^{2},yxy\right]=x_{1}^{4}=\left[x_{1}^{2},yxy\right],\\
1=\left[t,yxy\right]=\left[t,x^{4}\right],\:tx_{1}t^{-1}=\left[yxy,x_{1}\right].
\end{aligned}
\right.\right\rangle ,
\end{align*}

\noindent where $x_{1}=\left[yxy,x\right]$. Comparing the presentations
of $\widetilde{G}$ and $\widetilde{G_{1}}$, we have
\begin{align*}
\widetilde{G_{1}} & \cong\widetilde{G}/\left\langle \!\!\left\langle x_{1}^{4},\:\left[x_{1}^{2},yxy\right],\:\left[yxy,x_{1}\right]tx_{1}^{-1}t^{-1}\right\rangle \!\!\right\rangle _{\widetilde{G}},
\end{align*}

\noindent and define $\widetilde{\pi_{1}}:\widetilde{G}\to\widetilde{G_{1}}$
to be the quotient map.

By Lemma 3.1, we take the injection $\widetilde{\epsilon_{1}}:B_{4,2}\to\widetilde{G_{1}}$.
Then, the following diagram
\begin{align*}
\xymatrix{B_{4}\ar@{^{(}->}[d]^{\widetilde{\epsilon}}\ar[r]^{\beta_{4}\otimes\mathbb{F}_{2}} & B_{4,2}\ar@{^{(}->}[d]^{\widetilde{\epsilon_{1}}}\\
\widetilde{G}\ar[r]^{\widetilde{\pi_{1}}} & \widetilde{G_{1}}
}
\end{align*}

\noindent commutes, and we have the desired result from
\begin{align*}
\widetilde{\epsilon}\left(\ker\beta_{4}\otimes\mathbb{F}_{2}\right) & =\widetilde{\epsilon}\left(B_{4}\right)\cap\ker\widetilde{\pi_{1}}.\qed
\end{align*}

\noindent \end{prooftheorem13}

Lemma 3.1 follows from Lemma 3.2 and Lemma 3.4 word for word, just
as in the proof of Lemma 2.1 from Lemma 2.2 and Lemma 2.3.

\noindent \begin{lemma}

\noindent The group $B_{4,2}$ has the following presentation:
\begin{align*}
\left\langle x,y\;\left|\;\begin{aligned}1=x^{4}y^{-3}=\left[x^{2},yxy\right],\\
1=x_{i}^{4}=\left[x_{i}^{2},yxy\right],\:i=1,2,\cdots.
\end{aligned}
\right.\right\rangle .
\end{align*}

\noindent where $x_{i}$ is a word in $x$ and $y$ defined by (3).

\noindent \end{lemma}
\begin{proof}

\noindent Consider the following exact sequence:
\begin{align}
1 & \longrightarrow\left\langle \Delta^{2}\right\rangle \longrightarrow B_{4,2}\longrightarrow Q_{4,2}\longrightarrow1,
\end{align}

\noindent where we have simiplified $\beta_{4}\otimes\mathbb{F}_{2}\left(\Delta^{2}\right)$
as $\Delta^{2}$ by abusing notation. From (6), we compute a presentation
for $B_{4,2}$ from that of $\left\langle \Delta^{2}\right\rangle ,$
that of $Q_{4,2}$, and the conjugation of generators of $B_{4,2}$
on $\Delta^{2}$. We know $\left\langle \Delta^{2}\right\rangle \cong\mathbb{Z}$,
and the presentation for $Q_{4,2}$ in Lemma 2.2, which employ the
generators $x,y$ and the following relations:
\begin{lyxlist}{00.00.0000}
\item [{$\left.1_{o}\right)$}] \noindent $1=y^{3},$
\item [{$\left.1_{x}\right)$}] \noindent $1=x_{i}^{4},$ for each integer
$i\ge0,$
\item [{$\left.2_{x}\right)$}] \noindent $1=\left[x_{i}^{2},yxy\right],$
for each integer $i\ge0.$
\end{lyxlist}
$\quad\;\:$As we take the preimage $x$ and $y$ in $B_{4,2}$ as
usual (see the matrices at the beginning in Appendix A), we compute
in $B_{4,2}$ that
\begin{align}
 & x^{4}=\Delta^{2}=y^{3},\\
 & 1=\left[x^{2},yxy\right],
\end{align}

\noindent for $\left.1_{o}\right)$ and the relations for $i=0$ in
$\left.1_{x}\right)$ and $\left.2_{x}\right)$. For each word $R$
in $x$ and $y$ for $i>0$ in the right hand side of $\left.1_{x}\right)$
and $\left.2_{x}\right)$, by (6) itself,
\begin{align*}
R & =\Delta^{2k},
\end{align*}

\noindent for some integer $k$. As a word in commutators written
in $x$ and $y$, the determinant of $R$ is 1, and $k$ must be 0.
Thus, we obtain
\begin{align}
 & 1=x_{i}^{4},\;1\le i,\\
 & 1=\left[x_{i}^{2},yxy\right],\;1\le i.
\end{align}

Finally, since $\Delta^{2}\in Z\left(B_{4,2}\right)$,
\begin{align}
\left[x,\Delta^{2}\right] & =\left[y,\Delta^{2}\right].
\end{align}

In summary, we have a presentation of the generators $x,y,\Delta^{2}$
and the relations (7), (8), (9), (10) and (11). By the substitution
of $\Delta^{2}$ in (7), deleting the generator $\Delta^{2}$ and
the relations (7) and (11), and adding a relation
\begin{align*}
1 & =x^{4}y^{-3},
\end{align*}

we have constructed the presentation of Lemma 3.2.$\qedhere$

\noindent \end{proof}
\begin{remark}

The difference between Lemma 2.2 and Lemma 3.2 follows from the fact
that the element $x=\beta_{4}\left(\sigma_{1}\sigma_{2}\sigma_{3}\right)$
has a nontrivial determinant $t^{3}$ in $B_{4,2}$, but $\mathrm{PSL}\left(3,\,\mathbb{F}_{2}\left(t\right)\right)$
includes its image $p\left(x\right)=p\left(t^{-1}x\right)$ in $Q_{4,2}$.
On the other hand, the following isomorphism induced by the natural
surjection $p|_{\mathrm{SL}\left(3,\mathbb{F}_{2}\left(t\right)\right)}$
\begin{align*}
\mathrm{SL}\left(3,\,\mathbb{F}_{2}\left(t\right)\right) & \cong\mathrm{PSL}\left(3,\,\mathbb{F}_{2}\left(t\right)\right)
\end{align*}

\noindent assures that the relations in Lemma 3.2 written in the commutators
$x_{1},x_{2},\cdots$ are satisfied. Therefore, one can deduce Lemma
3.2 by rewriting the presentation for $B_{4,2}$ in Cooper\textendash Long
\citep[Corollary 4.9]{MR1431138} as in Appendix A. We presented here
a simpler proof, by which $\mathrm{Lemma}\;3.2$ is a direct corollary
of Lemma 2.2.

\noindent \end{remark}
\begin{lemma}

\noindent For each positive integer $i$, let us define a subgroup
$\widetilde{S_{i}}$ to be generated by two elements $x_{i}$ and
$yxy$ in $B_{4,2}$. Then, we define a group isomorphism
\begin{align*}
 & \widetilde{\eta_{ij}}:\widetilde{S_{i}}\to\widetilde{S_{j}},\\
 & \widetilde{\eta_{ij}}\left(x_{i}\right)=x_{j},\;\widetilde{\eta_{ij}}\left(yxy\right)=yxy,
\end{align*}

\noindent for every pair $\left(i,j\right)$.

\noindent \end{lemma}
\begin{proof}

\noindent As in Section 2, we use the linearity of $B_{4,2}$ to establish
Lemma 3.3. Recall for the matrices $x,y,M$ in $\mathrm{GL}\left(3,\,\mathbb{F}_{2}\left(t\right)\right)$,
as in the proof of Lemma 2.3, we obtain the following identities:
\begin{align*}
 & MxM^{-1}=t\left[yxy,x\right],\\
 & M\left(yxy\right)M^{-1}=yxy.
\end{align*}

Thus, expanding the commutator,
\begin{align*}
 & \,Mx_{1}M^{-1}\\
= & \,\left(yxy\right)^{-1}M\left(t^{-1}x\right)^{-1}M^{-1}\left(yxy\right)M\left(t^{-1}x\right)M^{-1}\\
= & \,\left[yxy,x_{1}\right].
\end{align*}

The rest of the proof is the same induction as that in the proof of
Lemma 2.3.$\qedhere$

\noindent \end{proof}
\begin{appendices}
	\section{Proof of Lemma 2.2}

\noindent The goal is to prove:

\noindent \begin{lemma22}

\noindent The group $Q_{4,2}$ is presented by the generators $x,y$
and the following relations:
\begin{lyxlist}{00.00.0000}
\item [{$\left.1_{o}\right)$}] \noindent $1=y^{3},$
\item [{$\left.1_{x}\right)$}] \noindent $1=x_{i}^{4},$ for each integer
$i\ge0,$
\item [{$\left.2_{x}\right)$}] \noindent $1=\left[x_{i}^{2},yxy\right],$
for each integer $i\ge0,$
\end{lyxlist}
\noindent where $x_{i}$ is a word in $x$ and $y$, recursively defined
as
\begin{align*}
x_{0} & :=x,\;x_{i+1}:=\left[yxy,x_{i}\right],\;i\ge0.
\end{align*}

\noindent \end{lemma22}

To prove Lemma 2.2, we need several auxiliary definitions and lemmas.
Throughout Appendix A we will only derive a set of relations from
another, and the sets supposed first will be given by Cooper\textendash Long
\citep{MR1431138}. Nevertheless, if a certain relation appears dubious,
a reader has the choice to check it through direct computation; here
are the matrix representations of $x$, $y$, and $yxy$ in $B_{4,2}\in\mathrm{GL}\left(3,\,\mathbb{F}_{2}\left[t,t^{-1}\right]\right)$.
\begin{align*}
x & =\beta_{4}\otimes\mathbb{F}_{2}\left(\sigma_{1}\sigma_{2}\sigma_{3}\right)=\left(\begin{array}{ccc}
0 & 0 & t\\
t & 0 & t\\
0 & t & t
\end{array}\right)=t\left(\begin{array}{ccc}
0 & 0 & 1\\
1 & 0 & 1\\
0 & 1 & 1
\end{array}\right),
\end{align*}

\begin{align*}
y & =\beta_{4}\otimes\mathbb{F}_{2}\left(\sigma_{1}\sigma_{2}\sigma_{3}\sigma_{1}\right)=\left(\begin{array}{ccc}
0 & 0 & t\\
t^{2} & t & t\\
0 & t & t
\end{array}\right)=t\left(\begin{array}{ccc}
0 & 0 & 1\\
t & 1 & 1\\
0 & 1 & 1
\end{array}\right),
\end{align*}

\begin{align*}
yxy & =\left(\begin{array}{ccc}
t^{4} & 0 & 0\\
t^{4} & t^{3}\left(1+t\right) & t^{4}\\
t^{4} & t^{3} & 0
\end{array}\right)=t^{4}\left(\begin{array}{ccc}
1 & 0 & 0\\
1 & 1+t^{-1} & 1\\
1 & t^{-1} & 0
\end{array}\right).
\end{align*}

We first review Cooper\textendash Long's presentation for finite subgroups
of $Q_{4,2}$. Borrowing notations used by Cooper\textendash Long
\citep{MR1431138}, we define a family of elements $\left\{ a_{i}\right\} _{i\ge0}$
in $Q_{4,2}$:
\begin{align*}
a_{i} & :=\left(yxy\right)^{i+1}x\left(yxy\right)^{-i-1},\;i=0,1,\cdots
\end{align*}

\noindent in $Q_{4,2}$, and a family of finite subgroups $\left\{ \mathrm{Stab}\left(7^{\left(i\right)}\right)\right\} _{i\ge0}$
of $Q_{4,2}$ as
\begin{align*}
\mathrm{Stab}\left(7^{\left(i\right)}\right) & :=\left\langle x,a_{0},a_{1},\cdots,a_{i}\right\rangle .
\end{align*}

\noindent \begin{theorem}

\noindent The group $\mathrm{Stab}\left(7^{\left(n\right)}\right)$
is presented by the generators $x,a_{0},a_{1},\cdots,a_{n}$ and the
following relations:
\begin{lyxlist}{00.00.0000}
\item [{$\left.1_{a}\right)$}] \noindent $1=x^{4},$
\item [{$\left.2_{a}\right)$}] \noindent $x^{2}=a_{0}^{2}=\cdots=a_{n}^{2},$
\item [{$\left.3_{a}\right)$}] \noindent $\left(xa_{i}\right)^{4}=1,\;i=0,1,\cdots,n,$
\item [{$\left.4_{a}\right)$}] \noindent $\left[a_{i},a_{j}\right]=\left[a_{i+k},a_{j+k}\right],\;0\le i<j,\;j+k\le n,$ 
\item [{$\left.5_{a}\right)$}] \noindent $\left[x,a_{i}\right]=\left[a_{k-1},a_{i+k}\right],\;0\le i,\;1\le k\le n-i.$
\end{lyxlist}
\noindent \end{theorem}
\begin{proof}

\noindent In \citep[Theorem 4.7]{MR1431138}, $\mathrm{Stab}\left(7^{\left(n\right)}\right)$
is presented by the generators $x,a_{0},a_{1},\cdots,a_{n}$ and the
relations $\left.1_{a}\right)\text{\textendash}\left.5_{a}\right)$
with a condition that $\mathrm{Stab}\left(7^{\left(n\right)}\right)$
is nilpotent of class 3. Thus, it suffices to show that the nilpotency
follows from $\left.1_{a}\right)\text{\textendash}\left.5_{a}\right)$,
which is established by Theorem A.4 later.$\qedhere$

\noindent \end{proof}

From now on, we temporarily suppose that $\mathrm{Stab}\left(7^{\left(n\right)}\right)$
is the group presented as Theorem A.1 without assuming the nilpotency
(see the proof above) until we prove Lemma A.2, Lemma A.3 and Theorem
A.4. The relations in Theorem A.1 for every $n$ as a whole consist
of an infinite sequence of relations of the presentation for $Q_{4,2}$,
which will be used in Theorem A.5.

As we will see, we need to analyze the center of $\mathrm{Stab}\left(7^{\left(n\right)}\right)$
to prove Theorem A.4. To do this, we transform the presentation of
Theorem A.1 by newly defined generators
\begin{align*}
b_{i+1} & :=x^{-1}a_{i},
\end{align*}

\noindent for each $i\ge0$. From now on, we say two group presentations
are \emph{equivalent} if the two presented groups are isomorphic.
To prove the equivalence, we rely on the Tietze transformation.

\noindent \begin{lemma}

\noindent The group $\mathrm{Stab}\left(7^{\left(n\right)}\right)$
is presented by the generators $x,b_{1},b_{2},\cdots,b_{n+1}$ and
the following relations:
\begin{lyxlist}{00.00.0000}
\item [{$\left.1_{b}\right)$}] \noindent $1=x^{4}=b_{i}^{4},\;i=1,\cdots,n+1,$
\item [{$\left.2_{b}\right)$}] \noindent $\left[x,b_{i}\right]=b_{i}^{2},\;i=1,\cdots,n+1,$
\item [{$\left.3_{b}\right)$}] \noindent $\left[b_{i},b_{j}\right]=b_{i}^{2}b_{j-i}^{2}b_{j}^{2},\;1\le i<j\le n+1.$
\end{lyxlist}
\noindent \end{lemma}
\begin{proof}

\noindent We prove that the presentation of Theorem A.1 is equivalent
to that of Lemma A.2, by substitution $b_{i+1}=x^{-1}a_{i}$.

At first, we suppose $\left.1_{a}\right)\text{\textendash}\left.5_{a}\right)$
and show $\left.1_{b}\right)\text{\textendash}\left.3_{b}\right)$.
From $\left.2_{a}\right)$, we directly see
\begin{equation}
x^{-1}b_{i}^{-1}=b_{i}x^{-1}.
\end{equation}

From (12), we obtain
\begin{align}
\left[x^{2},b_{i}\right] & =1,
\end{align}

\noindent for each $i$, since
\begin{align*}
\left[x^{2},b_{i}\right] & =x^{-2}b_{i}^{-1}x^{2}b_{i}=x^{-1}b_{i}xb_{i}=1.
\end{align*}

Thus, we deduce $b_{i}^{4}=1$ from $\left.3_{a}\right)$, and establish
$\left.1_{b}\right)$. Furthermore, we obtain
\begin{align}
\left[x,b_{i}^{2}\right] & =1.
\end{align}

Indeed, from (12) and $b_{i}^{4}=1$,
\begin{align*}
\left[x,b_{i}^{2}\right] & =x^{-1}b_{i}^{-2}xb_{i}^{2}=x^{-1}b_{i}^{-1}xb_{i}^{3}=x^{-1}b_{i}^{-1}xb_{i}^{-1}=1.
\end{align*}

Now, rearranging $\left.2_{a}\right)$ as $\left[x,b_{i}^{-1}\right]=b_{i}^{-2}$,
and using $\left.1_{b}\right)$ and (14), we obtain $\left.2_{b}\right)$:
\[
\left[x,b_{i}\right]=b_{i}^{2}.
\]

To rewrite $\left.5_{a}\right)$ in terms of $x$ and $b_{i}$, note
that
\begin{align*}
 & \left[x,xb_{i}\right]=\left[x,b_{i}\right],\\
 & \left[xb_{k},xb_{i+k+1}\right]=\left[b_{k},b_{i+k+1}x\right]\left[x,b_{i+k+1}\right]=\left[b_{k},x\right]x^{-1}\left[b_{k},b_{i+k+1}\right]x\left[x,b_{i+k+1}\right],
\end{align*}

\noindent where the second row follows from the commutator identity
$\left[a,bc\right]=\left[a,c\right]c^{-1}\left[a,b\right]c$. From
$\left.2_{b}\right)$, we rewrite $\left.5_{a}\right)$ as
\begin{align*}
\left[b_{k},b_{i+k+1}\right] & =x\left(b_{k}^{2}b_{i+1}^{2}b_{i+k+1}^{-2}\right)x^{-1},
\end{align*}

\noindent and by $\left.1_{b}\right)$ and (14), we finally obtain
$\left.3_{b}\right)$.

On the other hand, we suppose $\left.1_{b}\right)\text{\textendash}\left.3_{b}\right)$
and show $\left.1_{a}\right)\text{\textendash}\left.5_{a}\right)$.
The first thing we notice is the redundancy of $\left.4_{a}\right)$;
we already have two relations implying $\left.4_{a}\right)$ from
$\left.5_{a}\right)$:
\[
\left[x,a_{j-i-1}\right]=\left[a_{i},a_{j}\right],\;\left[x,a_{j-i-1}\right]=\left[a_{i+k},a_{j+k}\right].
\]

Under the conditions (13) and (14), doing the exact opposite of the
calculation we have done before, $\left.1_{a}\right)$ and $\left.3_{a}\right)$
follow from $\left.1_{b}\right)$; $\left.2_{a}\right)$ follows from
$\left.2_{b}\right)$; and $\left.5_{a}\right)$ follows from $\left.2_{b}\right)$
and $\left.3_{b}\right)$. Thus, we only need to show (13) and (14).
Expanding the commutator in the left hand side of (13), we have
\begin{align*}
 & \,\left[x^{2},b_{i}\right]\\
= & \,x^{-2}b_{i}^{-1}x^{2}b_{i}=x^{-1}\left[x,b_{i}\right]b_{i}^{-1}xb_{i}\\
= & \,x^{-1}b_{i}xb_{i}=b_{i}\left[b_{i},x\right]b_{i}=1,
\end{align*}

\noindent where the third equality and the fourth follow from $\left.2_{b}\right)$.
Expanding (14), we have
\begin{align*}
 & \,\left[x,b_{i}^{2}\right]\\
= & \,x^{-1}b_{i}^{-2}xb_{i}^{2}=x^{-1}b_{i}^{-1}x\left[x,b_{i}\right]b_{i}\\
= & \,x^{-1}b_{i}^{-1}xb_{i}^{3}=\left[x,b_{i}\right]b_{i}^{2}=b_{i}^{4}=1,
\end{align*}

\noindent which the third equality follows from $\left.2_{b}\right)$,
the fourth from $\left.2_{b}\right)$, and the last from $\left.1_{b}\right)$.$\qedhere$

\noindent \end{proof}

Lemma A.2 facilitates determining which elements are in the center.
Put
\begin{align*}
E_{2^{i}}
\end{align*}

\noindent as the elementary abelian group of order $2^{i}$.

\noindent \begin{lemma}

\noindent The center of $\mathrm{Stab}\left(7^{\left(n\right)}\right)$
is generated by $x^{2},b_{1}^{2},\cdots,b_{n+1}^{2}$, and isomorphic
to $E_{2^{n+2}}$. Moreover, it admits a short exact sequence:
\begin{align*}
1 & \longrightarrow E_{2^{n+2}}\longrightarrow\mathrm{Stab}\left(7^{\left(n\right)}\right)\longrightarrow E_{2^{n+2}}\longrightarrow1.
\end{align*}

\noindent In particular, for each element $g\in\mathrm{Stab}\left(7^{\left(n\right)}\right)$,
$g^{2}$ belongs to the center, and the exponent of $\mathrm{Stab}\left(7^{\left(n\right)}\right)$
is 4.

\noindent \end{lemma}
\begin{proof}

\noindent From the proof of Lemma A.2, the center includes $x^{2}$
from (13). At first, we need to show each element $b_{i}^{2}$ is
central.

\noindent \begin{claim1}

\noindent Each element $b_{i}^{2}$ such that $1\le i\le n+1$ belongs
to the center $Z\left(\mathrm{Stab}\left(7^{\left(n\right)}\right)\right)$.

\noindent \end{claim1}
\begin{proofclaim1}

\noindent We use induction on $n$. For $n=1$, the case is trivial
from (14). Suppose $b_{i}^{2}\in Z\left(\mathrm{Stab}\left(7^{\left(N\right)}\right)\right)$
for each $1\le i\le N+1$. Rearrange the equation $\left.3_{b}\right)$
as
\begin{align}
\left(b_{i}b_{N+2}^{-1}\right)^{2} & =b_{N-i+2}^{2}.
\end{align}

The identity (15) implies that
\begin{align*}
\left[b_{N-i+2}^{2},b_{N+2}b_{i}^{-1}\right] & =1,
\end{align*}

\noindent and by applying the commutator identity $\left[a,bc\right]=\left[a,c\right]c^{-1}\left[a,b\right]c$
to the commutator in the left hand side, the induction hypothesis
$\left[b_{N-i+2}^{2},b_{i}\right]=1$ implies that for each $1\le i\le N+1$,
\begin{align}
\left[b_{N-i+2}^{2},b_{N+2}\right] & =1.
\end{align}

Now, we only need to show that for each $1\le i\le N+1$,
\begin{equation}
\left[b_{N+2}^{2},b_{N-i+2}\right]=1.
\end{equation}

By the definition of the commutators, we expand
\begin{align*}
 & \,\left[b_{N+2}^{2},b_{N-i+2}\right]\\
= & \,b_{N+2}^{-2}b_{N-i+2}^{-1}b_{N+2}^{2}b_{N-i+2}=b_{N+2}^{-1}\left[b_{N+2},b_{N-i+2}\right]b_{N-i+2}^{-1}b_{N+2}b_{N-i+2}.
\end{align*}

From $\left.3_{b}\right)$,
\begin{align*}
 & \,b_{N+2}^{-1}\left[b_{N+2},b_{N-i+2}\right]b_{N-i+2}^{-1}b_{N+2}b_{N-i+2}\\
= & \,b_{N+2}b_{i}^{2}b_{N-i+2}b_{N+2}b_{N-i+2}.
\end{align*}

Finally, (16) and (15) imply
\begin{align*}
 & \,b_{N+2}b_{i}^{2}b_{N-i+2}b_{N+2}b_{N-i+2}\\
= & \,b_{i}^{2}\left(b_{N+2}b_{N-i+2}^{-1}\right)^{2}=1,
\end{align*}

\noindent which demonstrates (17). This concludes the proof of Claim
1.$\qed$

\noindent \end{proofclaim1}

We return to the proof of Lemma A.3. Now, it suffices to show
\begin{align*}
Z\left(\mathrm{Stab}\left(7^{\left(n\right)}\right)\right) & \subset\left\langle x^{2},b_{1}^{2},\cdots,b_{n+1}^{2}\right\rangle .
\end{align*}

Choose an element $g\in\mathrm{Stab}\left(7^{\left(n\right)}\right)$.
Since $\left\langle x^{2},b_{1}^{2},\cdots,b_{n+1}^{2}\right\rangle \subset Z\left(\mathrm{Stab}\left(7^{\left(n\right)}\right)\right)$,
we rearrange $g$ as
\begin{align*}
g & =C\left(g\right)x^{i_{0}}b_{1}^{i_{1}}b_{2}^{i_{2}}\cdots b_{n+1}^{i_{n+1}},
\end{align*}

\noindent where $C\left(g\right)\in\left\langle x^{2},b_{1}^{2},\cdots,b_{n+1}^{2}\right\rangle $
and $i_{j}\in\left\{ 0,1\right\} $ for $0\le j\le n+1$, by successive
application of $\left.2_{b}\right)$ and $\left.3_{b}\right)$. Therefore,
the centeralizer of $x$ is $\left\langle x,b_{1}^{2},\cdots,b_{n+1}^{2}\right\rangle $,
and that of $b_{1}$ is $\left\langle x^{2},b_{1},b_{2}^{2}\cdots,b_{n+1}^{2}\right\rangle $,
each of which has $\left\langle x^{2},b_{1}^{2},\cdots,b_{n+1}^{2}\right\rangle $
as a subgroup of index 2; thus the intersection $\left\langle x^{2},b_{1}^{2},\cdots,b_{n+1}^{2}\right\rangle $
contains the center.$\qedhere$

\noindent \end{proof}
\begin{theorem}

\noindent The derived subgroup of $\mathrm{Stab}\left(7^{\left(n\right)}\right)$
is generated by $b_{1}^{2},\cdots,b_{n+1}^{2}$, and isomorphic to
$E_{2^{n+1}}$. Moreover, it admits a short exact sequence:
\begin{align*}
1 & \longrightarrow E_{2^{n+1}}\longrightarrow\mathrm{Stab}\left(7^{\left(n\right)}\right)\longrightarrow\mathbb{Z}/4\mathbb{Z}\times E_{2^{n+1}}\longrightarrow1.
\end{align*}

\noindent In particular, the group $\mathrm{Stab}\left(7^{\left(n\right)}\right)$
is nilpotent of class 2 for each $n\ge0$.

\noindent \end{theorem}
\begin{proof}

\noindent By Lemma A.3, the subgroup $\left\langle b_{1}^{2},\cdots,b_{n+1}^{2}\right\rangle $
is central, and isomorphic to $E_{2^{n+1}}$. From the relation $\left.2_{b}\right)$,
it is contained in the derived subgroup. By the relations $\left.2_{b}\right)$
and $\left.3_{b}\right)$, the quotient
\[
\mathrm{Stab}\left(7^{\left(n\right)}\right)/\left\langle b_{1}^{2},\cdots,b_{n+1}^{2}\right\rangle 
\]

\noindent is abelian, which implies that $\left\langle b_{1}^{2},\cdots,b_{n+1}^{2}\right\rangle $
is exactly the derived subgroup.$\qedhere$

\noindent \end{proof}

Here we have established the nilpotency of $\mathrm{Stab}\left(7^{\left(n\right)}\right)$
based on the presentation given in Theorem A.1, as mentioned earlier.
In fact, our result on the nilpotency is stronger than what Cooper\textendash Long
assumed.

Now, let us consider the whole group $Q_{4,2}$. Cooper\textendash Long
\citep[Theorem 4.8]{MR1431138} gives a presentation for $Q_{4,2}$
that included the nilpotency of class 3 on the subgroup $\left\langle x,a_{i},i\ge0\right\rangle $.
However, we slightly modify the presentation by not requiring the
nilpotency explicitly, as we have done in Theorem A.4.

\noindent \begin{theorem}

\noindent The group $Q_{4,2}$ is presented by the generators $x,y$
and the following relations:
\begin{lyxlist}{00.00.0000}
\item [{$\left.1_{o}\right)$}] \noindent $1=y^{3},$
\item [{$\left.2_{o}\right)$}] \noindent $1=\left[x^{2},yxy\right],$
\item [{$\left.1_{b}\right)$}] \noindent $1=x^{4}=b_{i}^{4},$ for each
integer $i\ge1,$
\item [{$\left.2_{b}\right)$}] \noindent $\left[x,b_{i}\right]=b_{i}^{2},$
for each integer $i\ge1,$
\item [{$\left.3_{b}\right)$}] \noindent $\left[b_{i},b_{j}\right]=b_{i}^{2}b_{j-i}^{2}b_{j}^{2},\;1\le i<j,$
\end{lyxlist}
\noindent where $b_{i}$ is a word in $x$ and $y$, defined as $b_{i}=\left[x,\left(yxy\right)^{-i}\right]$
for each integer $i\ge1$.

\noindent \end{theorem}
\begin{proof}

\noindent The proof of \citep[Theorem 4.8]{MR1431138} showed that
the relations in a presentation for $Q_{4,2}$ with the generators
$x,y$ can be written as a word in $x,a_{0},a_{1},\cdots,a_{n}$ in
the group $\mathrm{Stab}\left(7^{\left(n\right)}\right)$ for some
sufficiently large $n$, except for $\left.1_{o}\right)$ and $\left.2_{o}\right)$.
By substitution $b_{i+1}=x^{-1}a_{i}$, we obtain the presentation
for $Q_{4,2}$ by gathering each relation of $\mathrm{Stab}\left(7^{\left(n\right)}\right)$
in Lemma A.2 for every $n$ with $\left.1_{o}\right)$ and $\left.2_{o}\right)$.$\qedhere$

\noindent \end{proof}
\begin{lemma}

\noindent The group $Q_{4,2}$ is presented by the generators $x,y$
and the following relations:
\begin{lyxlist}{00.00.0000}
\item [{$\left.1_{o}\right)$}] \noindent $1=y^{3},$
\item [{$\left.1_{b}\right)$}] \noindent $1=x^{4}=b_{i}^{4},$ for each
integer $i\ge1,$
\item [{$\left.2_{o}'\right)$}] \noindent $1=\left[x^{2},yxy\right]=\left[b_{i}^{2},yxy\right],$
for each integer $i\ge1,$
\end{lyxlist}
\noindent where $b_{i}$ is a word in $x$ and $y$, defined as $b_{i}=\left[x,\left(yxy\right)^{-i}\right]$
for each integer $i\ge1.$

\noindent \end{lemma}
\begin{proof}

\noindent Suppose the presentation of Theorem A.5. Then, we need only
to show $1=\left[b_{i}^{2},yxy\right]$. Note that
\begin{align}
\left(yxy\right)b_{i}\left(yxy\right)^{-1} & =yxyx^{-1}\left(yxy\right)^{-1}\left(yxy\right)^{i+1}x\left(yxy\right)^{-i-1}=b_{1}^{-1}b_{i+1}.
\end{align}

It suffices to show $\left[b_{i}^{2},\left(yxy\right)^{-1}\right]=1$.
We deduce from (18),
\begin{align*}
 & \,\left[b_{i}^{2},\left(yxy\right)^{-1}\right]\\
= & \,b_{i}^{-2}\left(yxy\right)b_{i}^{2}\left(yxy\right)^{-1}=b_{i}^{2}\left(b_{1}^{-1}b_{i+1}\right)^{2}.
\end{align*}

Since $\left.1_{b}\right)$ and Lemma A.3 imply $\left(b_{1}^{-1}b_{i+1}\right)^{2}=\left(b_{1}b_{i+1}\right)^{2}$,
we have
\begin{align*}
 & \,b_{i}^{2}\left(b_{1}^{-1}b_{i+1}\right)^{2}\\
= & \,b_{i}^{2}b_{1}b_{i+1}^{2}b_{1}\left[b_{1},b_{i+1}\right]\\
= & b_{i}^{2}b_{1}^{2}b_{i+1}^{2}\left[b_{1},b_{i+1}\right]=1,
\end{align*}

\noindent where the second eqality follows from Lemma A.3 and the
last follows from $\left.3_{b}\right)$, $\left.1_{b}\right)$, and
Lemma A.3.

Conversely, suppose $\left.1_{o}\right)$, $\left.1_{b}\right)$,
and $\left.2_{o}'\right)$. We need to show $\left.2_{b}\right)$
and $\left.3_{b}\right)$. For $\left.2_{b}\right)$,
\begin{align*}
 & \,\left[x,b_{i}\right]\\
= & \,x^{-1}\left[\left(yxy\right)^{-i},x\right]x\left[x,\left(yxy\right)^{-i}\right]=\left[x,\left(yxy\right)^{-i}\right]^{2}=b_{i}^{2},
\end{align*}

\noindent where $\left.2_{o}'\right)$ is used. Now, we claim (14)
$\left[b_{i}^{2},x\right]=1$ for each $i$. This claim is proved
as follows:
\begin{align*}
 & \,\left[b_{i}^{-2},x\right]\\
= & \,b_{i}^{2}x^{-1}\left[\left(yxy\right)^{-i},x\right]^{2}x=b_{i}^{2}\left[x,\left(yxy\right)^{-i}\right]^{2}=b_{i}^{4}=1,
\end{align*}

\noindent from $\left.2_{o}'\right)$.

Let us finally show $\left.3_{b}\right)$. For each pair $\left(i,j\right)$
such that $i<j$,
\begin{align*}
 & \,\left[b_{i},b_{j}\right]\\
= & \,b_{i}^{-1}\left[\left(yxy\right)^{-j},x\right]\left[x,\left(yxy\right)^{-i}\right]b_{j}\\
= & \,b_{i}^{-2}\left(x^{-1}\left(yxy\right)^{i}x\right)\left(\left(yxy\right)^{j-i}x^{-1}\left(yxy\right)^{i-j}x\right)^{2}\left(x^{-1}\left(yxy\right)^{-i}x\right)b_{j}^{2}\\
= & \,b_{i}^{2}b_{j-i}^{2}b_{j}^{2},
\end{align*}

\noindent where we used $\left.1_{b}\right)$, $\left.2_{o}'\right)$,
and the claim (14).$\qedhere$

\noindent \end{proof}
\begin{lemma}

\noindent The group $Q_{4,2}$ is presented by the generators $x,y$
and the following relations:
\begin{lyxlist}{00.00.0000}
\item [{$\left.1_{o}\right)$}] \noindent $1=y^{3},$
\item [{$\left.1_{c}\right)$}] \noindent $1=x^{4}=b_{-i}^{4},$ for each
integer $i\ge1,$
\item [{$\left.2_{c}\right)$}] \noindent $1=\left[x^{2},yxy\right]=\left[b_{-i}^{2},yxy\right],$
for each integer $i\ge1,$
\end{lyxlist}
\noindent where $b_{-i}$ is a word in $x$ and $y$, defined as $b_{-i}=\left[x,\left(yxy\right)^{i}\right]$.

\noindent \end{lemma}
\begin{proof}

\noindent The proof is direct from the relation between $b_{i}$ and
$b_{-i}$:
\begin{align*}
b_{-i} & =\left(yxy\right)^{-i}b_{i}^{-1}\left(yxy\right)^{i}.\qedhere
\end{align*}

\noindent \end{proof}

Now we are ready to prove Lemma 2.2. let us define $x_{i}$ as a word
in $x$ and $y$ recursively:
\begin{align*}
x_{0} & :=x,\;x_{i+1}:=\left[yxy,x_{i}\right],\;0\le i.
\end{align*}

\noindent \begin{prooflemma22}

The first step of the proof is to use the set of words $\left\{ x_{0},x_{1},\cdots,x_{n}\right\} $
to rewrite the other set of words $\left\{ x,b_{-1},b_{-2},\cdots,b_{-n}\right\} $,
and vice versa.

\noindent \begin{claim2}

\noindent For each $n\ge1$, $x_{n}\in\left\langle x,b_{-1},\cdots,b_{-n}\right\rangle $
and $b_{-n}\in\left\langle x_{0},x_{1},\cdots,x_{n}\right\rangle $.

\noindent \end{claim2}
\begin{proofclaim2}

\noindent We use induction. Since $b_{-1}=x_{1}^{-1}$ by definition,
the case $n=1$ is trivial. Suppose $x_{i}\in\left\langle x,b_{-1},\cdots,b_{-i}\right\rangle $
for $1\le i\le N$. By definition,
\begin{align*}
x_{N+1} & =\left(yxy\right)^{-1}x_{N}^{-1}\left(yxy\right)x_{N},
\end{align*}

\noindent and from
\begin{align*}
\left(yxy\right)^{-1}b_{-i}\left(yxy\right) & =\left(xb_{-1}\right)^{-1}\left(xb_{-i-1}\right)=b_{-1}^{-1}b_{-i-1},
\end{align*}

\noindent we obtain $x_{N+1}\in\left\langle x,b_{-1},\cdots,b_{-N-1}\right\rangle $
by the induction hypothesis. Conversely, suppose 

\noindent $b_{-i}\in\left\langle x_{0},x_{1},\cdots,x_{i}\right\rangle $
for $1\le i\le N$. By definition, for $i>1$,
\begin{align*}
b_{-i+1}^{-1}b_{-i} & =\left(yxy\right)^{-i+1}\left[x,yxy\right]\left(yxy\right)^{i-1},
\end{align*}

\noindent and we obtain $\left(yxy\right)^{-i+1}\left[x,yxy\right]\left(yxy\right)^{i-1}\in\left\langle x_{0},x_{1},\cdots,x_{i}\right\rangle $
for $1\le i\le N$ by the induction hypothesis. From
\begin{align*}
\left(yxy\right)^{-1}x_{i}\left(yxy\right) & =x_{i}x_{i+1}^{-1},
\end{align*}

\noindent we conclude that $\left(yxy\right)^{-N}\left[x,yxy\right]\left(yxy\right)^{N}\in\left\langle x_{0},x_{1},\cdots,x_{N+1}\right\rangle $.$\qed$

\noindent \end{proofclaim2}

\noindent From here, we will prove the presentation of Lemma A.7 with
the relations:
\begin{lyxlist}{00.00.0000}
\item [{$\left.1_{o}\right)$}] \noindent $1=y^{3},$
\item [{$\left.1_{c}\right)$}] \noindent $1=x^{4}=b_{-i}^{4},$ for each
integer $i\ge1,$
\item [{$\left.2_{c}\right)$}] \noindent $1=\left[x^{2},yxy\right]=\left[b_{-i}^{2},yxy\right],$
for each integer $i\ge1,$
\end{lyxlist}
\noindent is equivalent to that of Lemma 2.2 with the relations:
\begin{lyxlist}{00.00.0000}
\item [{$\left.1_{o}\right)$}] \noindent $1=y^{3},$
\item [{$\left.1_{x}\right)$}] \noindent $1=x_{i}^{4},$ for each integer
$i\ge0,$
\item [{$\left.2_{x}\right)$}] \noindent $1=\left[x_{i}^{2},yxy\right],$
for each integer $i\ge0.$
\end{lyxlist}
$\quad\;\:$Suppose the former. Claim 2 guarantees that each element
$x_{i+1}$ is similar to an element $g_{i+1}$ in $\mathrm{Stab}\left(7^{\left(i\right)}\right)$
by some power of $yxy$ as in the proof of $\mathrm{Lemma\;A.7}$.
By Lemma A.3, the exponent of $\mathrm{Stab}\left(7^{\left(i\right)}\right)$
is 4, which implies $\left.1_{x}\right)$. Moreover, we have $g_{i+1}^{2}\in Z\left(\mathrm{Stab}\left(7^{\left(i\right)}\right)\right)$
by Lemma A.3. Thus, $\left.2_{x}\right)$ follows from $\left.2_{c}\right)$
by a repetitive application of the commutator identity
\begin{align*}
\left[a,bc\right] & =\left[a,c\right]c^{-1}\left[a,b\right]c.
\end{align*}

Conversely, suppose $\left.1_{x}\right)$ and $\left.2_{x}\right)$.
Before deriving the relations $\left.1_{c}\right)$ and $\left.2_{c}\right)$,
we claim:

\noindent \begin{claim3}

\noindent The relations $\left.1_{x}\right)$ and $\left.2_{x}\right)$
imply that
\begin{lyxlist}{00.00.0000}
\item [{$\left.3_{x}\right)$}] \noindent $\left[x_{i},x_{j}\right]\in\left\langle x_{0}^{2},x_{1}^{2},\cdots,x_{j}^{2}\right\rangle $,
for each pair $\left(i,j\right)$ such that $0\le i,j$,
\item [{$\left.4_{x}\right)$}] \noindent $\left[x_{i}^{2},x_{j}\right]=1$,
for each pair $\left(i,j\right)$ such that $0\le i,j$.
\end{lyxlist}
\noindent \end{claim3}
\begin{proofclaim3}

\noindent We use induction. Suppose $\left.3_{x}\right)$ and $\left.4_{x}\right)$
hold for pairs $\left(i,j\right)$ such that 

\noindent $0\le i,j\le N$. We first show $\left.3_{x}\right)$. When
$i\le N$, we deduce
\begin{equation}
\left[x_{i},x_{N+1}\right]=\left[x_{i},x_{N}\right]x_{N}^{-2}\left(yxy\right)^{-1}\left(x_{N-1}^{-1}\cdots x_{i+1}^{-1}x_{i}^{-1}\right)x_{N}\left(yxy\right)x_{i}x_{N+1},
\end{equation}

\noindent which is shown by the definition of commutators:
\begin{align*}
 & \,\left[x_{i},x_{N+1}\right]\\
= & \,x_{i}^{-1}x_{N}^{-1}\left(yxy\right)^{-1}x_{N}\left(yxy\right)x_{i}x_{N+1}\\
= & \,\left[x_{i},x_{N}\right]x_{N}^{-1}x_{i}^{-1}\left(yxy\right)^{-1}x_{N}\left(yxy\right)x_{i}x_{N+1}\\
= & \,\left[x_{i},x_{N}\right]x_{N}^{-1}x_{i+1}^{-1}\left(yxy\right)^{-1}x_{i}^{-1}x_{N}\left(yxy\right)x_{i}x_{N+1}\\
= & \,\left[x_{i},x_{N}\right]x_{N}^{-2}\left(yxy\right)^{-1}\left(x_{N-1}^{-1}\cdots x_{i+1}^{-1}x_{i}^{-1}\right)x_{N}\left(yxy\right)x_{i}x_{N+1}.
\end{align*}

In the right hand side of (19), by using the induction hypothesis
$\left.4_{x}\right)$ for each $j$ such that $i\le j\le N-1$, we
see
\[
x_{j}^{-1}x_{N}=x_{N}x_{j}^{-1}\left[x_{j}^{-1},x_{N}\right]=x_{N}x_{j}^{-1}\left[x_{j},x_{N}\right],
\]

\noindent so we pull $x_{N}$ before $\left(yxy\right)$ leftward,
leaving a product of commutators
\begin{align*}
\left[x_{N-1},x_{N}\right]\cdots\left[x_{i+1},x_{N}\right]\left[x_{i},x_{N}\right],
\end{align*}

\noindent which is included in the subgroup $\left\langle x_{0}^{2},x_{1}^{2},\cdots,x_{j}^{2}\right\rangle $
by $\left.3_{x}\right)$. Thus, by $\left.4_{x}\right)$ there exists
an element $C_{N}\in\left\langle x_{0}^{2},x_{1}^{2},\cdots,x_{N}^{2}\right\rangle $
such that
\begin{equation}
\left[x_{i},x_{N+1}\right]=C_{N}\left(yxy\right)^{-1}\left(x_{N}^{-1}x_{N-1}^{-1}\cdots x_{i+1}^{-1}x_{i}^{-1}\right)\left(yxy\right)x_{i}x_{N+1}.
\end{equation}

Finally, we simplify the right hand side of (20) as
\begin{align*}
 & \,C_{N}\left(yxy\right)^{-1}x_{N}^{-1}x_{N-1}^{-1}\cdots x_{i+1}^{-1}x_{i}^{-1}\left(yxy\right)x_{i}x_{N+1}\\
= & \,C_{N}\left(yxy\right)^{-1}x_{N}^{-1}x_{N-1}^{-1}\cdots x_{i+1}^{-1}\left(yxy\right)x_{i+1}x_{i}^{-1}x_{i}x_{N+1}\\
= & \,C_{N}\left(yxy\right)^{-1}x_{N}^{-1}x_{N-1}^{-1}\cdots x_{i+2}^{-1}\left(yxy\right)x_{i+2}x_{i+1}^{-1}x_{i+1}x_{N+1}\\
= & \,C_{N}\left(yxy\right)^{-1}x_{N}^{-1}\left(yxy\right)x_{N}x_{N+1}\\
= & \,C_{N}\left(yxy\right)^{-1}\left(yxy\right)x_{N+1}^{2}\\
= & \,C_{N}x_{N+1}^{2},
\end{align*}

\noindent which means $\left[x_{i},x_{N+1}\right]\in\left\langle x_{0}^{2},x_{1}^{2},\cdots,x_{N}^{2},x_{N+1}^{2}\right\rangle $
and establishes $\left.3_{x}\right)$ when $j=N+1$. The case $i=N+1$
also follows from the commutator identity $\left[x_{N+1},x_{j}\right]=\left[x_{j},x_{N+1}\right]^{-1}$.

For $\left.4_{x}\right)$, for each $i$ such that $0\le i\le N+1$,
\begin{align}
\left[x_{N+1},x_{i}^{2}\right] & =1=\left[x_{i}^{2},x_{N+1}^{2}\right]
\end{align}

\noindent directly follows from the definition of $x_{N+1}$, $\left.2_{x}\right)$,
and $\left.4_{x}\right)$ up to $i,j\le N$. We need to show that
\begin{equation}
\left[x_{N+1}^{2},x_{i}\right]=1.
\end{equation}

By the definition of the commutators we deduce that
\begin{align*}
 & \,\left[x_{N+1}^{2},x_{i}\right]\\
= & \,x_{N+1}^{-2}x_{i}^{-1}x_{N+1}^{2}x_{i}=x_{N+1}^{-1}\left[x_{N+1},x_{i}\right]x_{i}^{-1}x_{N+1}x_{i},
\end{align*}

\noindent and by the fact $\left[x_{N+1},x_{i}\right]\in\left\langle x_{0}^{2},x_{1}^{2},\cdots,x_{N}^{2},x_{N+1}^{2}\right\rangle $
and (21), we move the commutator $\left[x_{N+1},x_{i}\right]$ leftwards:

\noindent 
\begin{align*}
 & \,x_{N+1}^{-1}\left[x_{N+1},x_{i}\right]x_{i}^{-1}x_{N+1}x_{i}\\
= & \,\left[x_{N+1},x_{i}\right]x_{N+1}^{-1}x_{i}^{-1}x_{N+1}x_{i}\\
= & \,\left[x_{N+1},x_{i}\right]\left[x_{N+1},x_{i}\right].
\end{align*}

Because the induction hypothesis $\left.4_{x}\right)$ and (21) imply
that the subgroup $\left\langle x_{0}^{2},x_{1}^{2},\cdots,x_{N}^{2},x_{N+1}^{2}\right\rangle $
is abelian, we obtain (22) from $\left.1_{x}\right)$. This proves
Claim 3.$\qed$

\noindent \end{proofclaim3}

Finally, we conclude the proof of Lemma 2.2. From $\left.4_{x}\right)$
of Claim 3, each element $x_{j}^{2}$ is central in $\left\langle x_{0},x_{1},\cdots,x_{i+1}\right\rangle $.
Thus, for any element $g\in\left\langle x_{0},x_{1},\cdots,x_{i+1}\right\rangle $,
repetitively applying $\left.3_{x}\right)$ of Claim 3, we get a normal
form for $g$:
\begin{align*}
g & =C\left(g\right)x_{0}^{j_{0}}x_{1}^{j_{1}}\cdots x_{i+1}^{j_{i+1}},
\end{align*}

\noindent where $C\left(g\right)\in\left\langle x_{0}^{2},x_{1}^{2},\cdots,x_{i+1}^{2}\right\rangle $
and $j_{k}\in\left\{ 0,1\right\} $ for $0\le k\le i+1$. Thus, we
obtain $g^{4}=1$ and $g^{2}\in\left\langle x_{0}^{2},x_{1}^{2},\cdots,x_{i+1}^{2}\right\rangle $,
which establishes $\left.1_{c}\right)$ and $\left.2_{c}\right)$
by invoking Claim 2.$\qed$

\noindent \end{prooflemma22}
\end{appendices}

\begin{spacing}{0.9}
\bibliographystyle{amsplain}
\phantomsection\addcontentsline{toc}{section}{\refname}\bibliography{bibgen}

\providecommand{\bysame}{\leavevmode\hbox to3em{\hrulefill}\thinspace}
\providecommand{\MR}{\relax\ifhmode\unskip\space\fi MR }
\providecommand{\MRhref}[2]{%
  \href{http://www.ams.org/mathscinet-getitem?mr=#1}{#2}
}
\providecommand{\href}[2]{#2}
\begin{thebibliography}{10}

\bibitem{MR4478033}
James Belk, James Hyde, and Francesco Matucci, \emph{Embedding {$\Bbb{Q}$} into a finitely presented group}, Bull. Amer. Math. Soc. (N.S.) \textbf{59} (2022), no.~4, 561--567. \MR{4478033}

\bibitem{MR1725480}
Stephen Bigelow, \emph{The {B}urau representation is not faithful for {$n=5$}}, Geom. Topol. \textbf{3} (1999), 397--404. \MR{1725480}

\bibitem{MR3323579}
Tara Brendle, Dan Margalit, and Andrew Putman, \emph{Generators for the hyperelliptic {T}orelli group and the kernel of the {B}urau representation at {$t=-1$}}, Invent. Math. \textbf{200} (2015), no.~1, 263--310. \MR{3323579}

\bibitem{MR3660095}
Matthieu Calvez and Tetsuya Ito, \emph{A {G}arside-theoretic analysis of the {B}urau representations}, J. Knot Theory Ramifications \textbf{26} (2017), no.~7, 1750040, 29. \MR{3660095}

\bibitem{MR2629766}
Thomas Church and Benson Farb, \emph{Infinite generation of the kernels of the {M}agnus and {B}urau representations}, Algebr. Geom. Topol. \textbf{10} (2010), no.~2, 837--851. \MR{2629766}

\bibitem{MR1431138}
D.~Cooper and D.~D. Long, \emph{A presentation for the image of {${\rm Burau}(4)\otimes Z_2$}}, Invent. Math. \textbf{127} (1997), no.~3, 535--570. \MR{1431138}

\bibitem{MR1668343}
\bysame, \emph{On the {B}urau representation modulo a small prime}, The {E}pstein birthday schrift, Geom. Topol. Monogr., vol.~1, Geom. Topol. Publ., Coventry, 1998, pp.~127--138. \MR{1668343}

\bibitem{dlugie2022burau}
Ethan Dlugie, \emph{The {B}urau representation and shapes of polyhedra}, 2022. arXiv:2210.06561.

\bibitem{MR0251712}
E.~A. Gorin and V.~Ja. Lin, \emph{Algebraic equations with continuous coefficients, and certain questions of the algebraic theory of braids}, Mat. Sb. (N.S.) \textbf{78 (120)} (1969), 579--610. \MR{0251712}

\bibitem{MR130286}
G.~Higman, \emph{Subgroups of finitely presented groups}, Proc. Roy. Soc. London Ser. A \textbf{262} (1961), 455--475. \MR{130286}

\bibitem{MR2435235}
Christian Kassel and Vladimir Turaev, \emph{Braid groups}, Graduate Texts in Mathematics, vol. 247, Springer, New York, 2008, With the graphical assistance of Olivier Dodane. \MR{2435235}

\bibitem{MR2175118}
Sang~Jin Lee and Won~Taek Song, \emph{The kernel of {${\rm Burau}(4)\otimes Z_p$} is all pseudo-{A}nosov}, Pacific J. Math. \textbf{219} (2005), no.~2, 303--310. \MR{2175118}

\bibitem{MR545151}
N.~F. Smythe, \emph{The {B}urau representation of the braid group is pairwise free}, Arch. Math. (Basel) \textbf{32} (1979), no.~4, 309--317. \MR{545151}

\end{thebibliography}

\end{spacing}

$ $

{\small{}Donsung Lee; \href{mailto:disturin@snu.ac.kr}{disturin@snu.ac.kr}}{\small\par}

{\small{}Department of Mathematical Sciences and Research Institute
of Mathematics,}{\small\par}

{\small{}Seoul National University, Gwanak-ro 1, Gwankak-gu, Seoul,
South Korea 08826}{\small\par}

\clearpage{}

\pagebreak{}

\pagenumbering{arabic}

\renewcommand{\thefootnote}{A\arabic{footnote}}
\renewcommand{\thepage}{A\arabic{page}}
\renewcommand{\thetable}{A\arabic{table}}
\renewcommand{\thefigure}{A\arabic{figure}}

\setcounter{footnote}{0} 
\setcounter{section}{0}
\setcounter{table}{0}
\setcounter{figure}{0}
\end{document}